\documentclass{amsart}
\usepackage{bbm}
\usepackage{amsmath,amsthm,amsfonts,amssymb,amscd,mathrsfs}
\usepackage{psfrag,graphicx}
\usepackage{epic,eepic}
\usepackage{color}
\usepackage{ebezier}
\usepackage{enumerate}

\newtheorem{Thm}{Theorem}[section]
\newtheorem{Lem}[Thm]{Lemma}
\newtheorem{Prop}[Thm]{Proposition}
\newtheorem{Cor}[Thm]{Corollary}
\theoremstyle{remark}
\newtheorem{Def}[Thm] {Definition}
\newtheorem{Rem}[Thm] {Remark}

\newcommand{\diam}{\operatorname{diam}}

\newcommand{\length}{\operatorname{\length}}

\def\length{\operatorname{length}}

\keywords{Chain hyperbolicity, Closing property, 
Hyperbolic periodic measure}
\title[hyperbolic periodic points for $CH^1(p)$ ]{Hyperbolic periodic points for  chain hyperbolic homoclinic classes}
\author
{Wenxiang Sun$^{\dag}$, Yun Yang$^{*}$}

 \thanks{ $^{\dag}$ Sun is supported
by NNSFC(\# 10831003) and Doctoral Education Foundation of
China}
\address{$^{*, \dag}$ School of Mathematical Sciences,
Peking University, Beijing 100871, China; Email: $^{\dag}$sunwx@math.pku.edu.cn, $^{*}$yangy88@math.pku.edu.cn.}
\thanks{$^{*}$ Communication author}
\begin{document}
\maketitle

\def\abstractname{\textbf{Abstract}}
\begin{abstract}
In this paper we establish a closing
property and a hyperbolic closing property  for  thin
trapped chain hyperbolic homoclinic classes with one dimensional
center in partial hyperbolicity setting.    Taking advantage of theses
properties, we prove that the growth rate of the number of
hyperbolic periodic points is equal to the topological entropy.  We also
obtain that the hyperbolic periodic measures are dense in the space of invariant measures.
\end{abstract}

\section{Introduction}
Chain hyperbolic homoclinic classes are a kind of homoclinic classes with dominated splitting and some topological hyperbolicity,  which is put forward by S. Crovisier and E. Pujals in \cite{C2}. 
 They  present the fact that 
homoclinic classes of generic diffeomorphisms far from
homoclinic tangencies and heterodimensional cycles are
chain hyperbolic homoclinic classes \cite{C1,C2} . Thus chain hyperbolic homoclinic classes are very common in diffeomorphisms far away from homoclinic bifurcations. 
In contrast with the uniform hyperbolic part in the Palis conjecture \cite{Palis1,Palis2}  for the whole dynamics,  chain hyperbolic homoclinic class is a generalization of the basic set of Axiom A diffeomorphisms. 
  
The connection between the growth rate of the number of periodic points and the topological entropy of dynamical systems  attracts people's attention for a long time. 
In his paper \cite{Bowen}, R. Bowen proved that  the  growth rate of
the hyperbolic periodic points is equal to the topological entropy for a basic set of Axiom A
diffeomorphisms. 
Since chain
hyperbolic homoclinic class generalizes the  basic
set of Axiom A diffeomorphism, a natural proposal is to consider the  hyperbolic periodic points  for chain hyperbolic
systems. However, this is not a trivial question. The difficulties arise  not only in proving the existence of sufficiently many periodic points but also in  controlling the Lyapunov exponents  of sufficiently many  hyperbolic periodic points. To solve these difficulties, using different methods, we shall prove some new kind of closing properties for the chain hyperbolic systems in our setting.  
 
In this paper, we require that the dominated splitting of the chain hyperbolic homoclinic class is partially hyperbolic with one dimensional center.  We prove that the growth rate of the number of hyperbolic periodic points is equal to the topological entropy for  thin trapped chain hyperbolic homoclinic classes with one dimensional center.  

Now we start to state our results precisely. Before formulating the definition of chain hyperbolic homoclinic class,  we recall some useful  terminology first. 
Let $f$ be a diffeomorphism on a compact manifold $M$  and $\Lambda$ be a compact invariant set. A {\it dominated splitting} on $\Lambda$ is an invariant decomposition of $T_{\Lambda}M=E\oplus F$ such that for some integer
$N\geq 1$, any unit vectors $u_{x}\in E_{x}, v_{x}\in F_{x}$, at any point $x\in \Lambda$  the following is satisfied:
$$2\parallel Df^{N}u_{x}\parallel\leq \parallel Df^{N}v_{x}\parallel.$$ A {\it partially hyperbolic splitting} on $\Lambda$ is a dominated splitting $T_{\Lambda}M=E^{s}\oplus E^{c}\oplus E^{u}$ such that
$E^{s}$ and $E^{u}$ uniformly contacted by $f$ and $f^{-1}$ respectively.
A {\it plaque family} tangent to bundle $E$ is a family of continuous maps $\mathcal{W}$ from the linear bundle $E$ over $\Lambda$
into $M$ satisfying:
\begin{itemize}\item
for each $x\in \Lambda$, the induced map $\mathcal{W}_{x}:E_{x}\rightarrow M$ is a $\mathcal{C}^1$- embedding that satisfies $\mathcal{W}_{x}(0)=x$
and whose image is tangent to $E_{x}$ at $x$;
 \item
$(\mathcal{W}_{x})_{x\in \Lambda}$ is a continuous family of $\mathcal{C}^1$- embeddings.
\end{itemize}
Let $\mathcal{W}(x)$ be the image of embedding $\mathcal{W}_x$  at point $x$. A plague family $\mathcal{W}$ is {\it locally invariant} if there
exists $\rho>0$ such that for each $x\in\Lambda$ the image of the ball
$B(0,\rho)\subseteq E_{x}$ by $f\circ\mathcal{W}_{x}$ is contained
in the plaque $\mathcal{W}(f(x))$. 

Plaque family theorem
(Theorem 5.5 in \cite{HPS}) says that there always exists a locally invariant plaque family tangent to
$E$ (but is not unique in general).  Moreover, from the local integrable property, the images of  invariant plaque family  on compact manifold
have uniform length.  
 A plaque family is  called {\it trapped} if for each $x\in \Lambda$, one has
$f(\overline{\mathcal{W}(x)})\subseteq \mathcal{W}(f(x))$. It is {\it thin trapped} if for any neighborhood $U$ of the section $0$ in $E$ there exist:
\begin{itemize}\item a continuous family $(\varphi_{x})_{x\in\Lambda}$ of $\mathcal{C}^1$-diffeomorphisms of the spaces $(E_{x})_{x\in\Lambda}$ supported in $U$;
 \item a constant $\rho>0$, such that for any $x\in\Lambda$, one has
$f(\overline{\mathcal{W}_{x}\circ\varphi_{x}(B(0,\rho))})\subset
\mathcal{W}_{f(x)}\circ \varphi_{f(x)}(B(0,\rho))$.
\end{itemize}
Thus it can be seen that ``thin trapped" means that there exist nested families of trapped plaques whose diameters are arbitrarily small. 
\smallskip

Let $p$ be a hyperbolic periodic point for $f$. The {\it homoclinic class} $H(p)$ is defined as the closure set 
of  the collection of  transversal intersection points of the unstable manifolds $W^{u}(O)$ and the stable manifolds $W^{s}(O)$ of the  orbit $O$ of $p$. 
$H(p)$ also coincides with the closure set of the collection of
hyperbolic  periodic points $q$ that are homoclinically related to the orbit
of $p$, i.e. for some $i$, $W^{u}(q)$ and $W^{s}(q)$ have transversal intersection
points with the stable manifolds $W^{s}(f^i(p))$ and the unstable manifolds $W^s(f^i(p))$ of  point $f^i(p)$
respectively.

We are now ready for the definition of chain hyperbolic homoclinic class. 
\begin{Def}\label{chain} A homoclinic class $H(p)$ is called  {\it chain hyperbolic} if:
\begin{itemize}\item
$H(p)$ has a dominated splitting $T_{H(p)}M=E^{cs}\oplus E^{cu}$;
 \item there exist a plaque family $(\mathcal{W}^{cs}_{x})_{x\in H(p)}$
tangent to $E^{cs}$ which is trapped by $f$ and a plaque family
$(\mathcal{W}^{cu}_{x})_{x\in H(p)}$ tangent to $E^{cu}$ which is
trapped by $f^{-1}$;
 \item there exists a hyperbolic periodic point $q_{s} (\text{resp. }  q_{u})$ homoclinically related to the orbit of $p$ whose stable manifold contains $\mathcal{W}^{cs}(q_s)$(resp. whose unstable manifold contains $\mathcal{W}^{cu}(q_u)$) .$ $
\end{itemize}

\end{Def}
 In this paper, we study only the case when the dominated splitting of the chain hyperbolic homoclinic class is partially hyperbolic  with one dimensional center.  We formulate our objects in the following definition. 
\begin{Def}\label{CH^1(p)}
 Let $CH^1(p)$  be  a chain hyperbolic homoclinic class with dominated splitting $$T_{CH^1(p)}M=E^{cs}\oplus E^u=E^{s}\oplus E^c\oplus E^{u}$$ where $E^{u}$ is a uniformly expansive bundle,  $E^{s}$  is
a uniformly contracting bundle, the plaque family
$\mathcal{W}^{cs}$ associated with $E^{s}\oplus E^c$ is thin trapped for $f$ and $\dim E^c=1$.
\end{Def}

Denote the set of periodic points with period $n$ and the center Lyapunov exponents less than $-a$ by
\begin{eqnarray*}  P^a_n(f|_{CH^1(p)})  &=&\{x\in CH^1(p)\,|\,f^n(x)=x, \lambda(x)\geq a\geq 0 ,\\
&& \text{where } \lambda(x) \text{ satisfies } \|Df^{n}|_{E^{c}(f^i(x))}\|= e^{-\lambda(x)n}\}.
\end{eqnarray*}
As $\dim E^c=1$, $\lambda(x)$ is exactly the Lyapunov exponent of $x$ along $E^c$. It is worth to note that this is not true for $\dim E^c\geq2$, because the norm of a matrix  may be strictly greater than the eigenvalues of the matrix. 
Our main theorems are the following ones. 
\begin{Thm}\label{growth of periodic points } Let $f$ be a diffeomorphism on a compact manifold $M$  with  $CH^1(p)$. Then, we have $$\lim_{a\rightarrow 0}\limsup_{n\rightarrow +\infty}\frac{1}{n}\log
\sharp   P^a_n(f|_{CH^1(p)})= h_{\text{top}}(f|_{CH^1(p)}).$$
\end{Thm}

 Related to Theorem \ref{growth of periodic points } is a result by Katok \cite{K} that for $\mathcal{C}^{1+\alpha}$ diffeomorphisms on surface manifolds,  the growth rate of the number of  periodic points is greater than or equal to the topological entropy.      Based on Katok's construction of horseshoes for diffeomorphisms on surface manifolds, Chung and Hirayama \cite{CH} proved that the topological entropy of a $\mathcal{C}^{1+\alpha}$ surface diffeomorphism is given by the growth rate of the number of periodic points of saddle type.  Usually, there may not be sufficiently many hyperbolic periodic points. The observation in Theorem \ref{growth of periodic points } provides  us with many hyperbolic periodic points. \smallskip 

The following closing properties serve as the main tools to find periodic points.  The interest of seeking for hyperbolic periodic points leads us to the  definition, hyperbolic closing property. \begin{Def}\label{Specification}Let $f$ be a diffeomorphism on a compact manifold $M$ and let $K$ be  an compact invariant subset of $M$. The following property is called
{\it a closing property} : for any $\varepsilon>0$ small enough, there
exists $M(\varepsilon)\in\mathbb{N}$ such that for any two segments $x, f^n(x)\in K$ and  $y, f^m(y)\in K$,   there exists a periodic point $z\in K$ with period $P(z)$ such that 
\begin{enumerate}
\item $d(f^i(z), f^i(x))\leq \varepsilon, 0\leq i\leq n$;
\item $d(f^{n+k+j}(z), f^j(y))\leq \varepsilon, 0\leq j\leq m$ for some $0\leq k\leq M(\varepsilon)$;
\item $n+m+k\leq P(z)\leq n+m+2M(\varepsilon)$.
\end{enumerate}
Moreover, the parallel result holds also for more than two segments.
\end{Def}
This closing property is similar to the specification property for Axiom A diffeomorphisms (\cite{S}) when $y=f^{n+k}(x)$.  From the construction of the shadowing periodic point $z$ in section 3, it follows that when $y=x$ and $m=n$, the period of $z$ can be $n+k$, which is similar to the classical closing lemma. The only difference from the classical closing lemma is that there is a spare segment.   Nevertheless, the closing property in Definition \ref{Specification} can also be used to prove Corollary \ref{distribution of periodic measures}. 

 \begin{Def}\label{Specification2}Let $f$ be a diffeomorphism on a compact manifold $M$  with  $CH^1(p)$. Let $-\lambda$ be the largest negative  Lyapunov exponent of $p$. The following property is called
{\it a hyperbolic closing property} : for  any $0<a<1$ and any $\varepsilon>0$  small enough, there
exists $M(\varepsilon)\in\mathbb{N}$ such that for any two points in the same orbit $x, f^n(x)\in CH^1(p)$,  there exists a hyperbolic periodic point $z\in CH^1(p)$ with period $P(z)$ such that 
\begin{enumerate}
\item  $d(f^i(z), f^i(x))\leq \varepsilon, \forall\, 0\leq i\leq n$ and  $\frac{n}{1-a}\leq P(z)\leq \frac{n}{1-a}+2M(\varepsilon),$
 \item the Lyapunov exponent along the central direction $E^c$ of periodic point $z$ is smaller than $-\frac{a\lambda}{4}.$
\end{enumerate}
\end{Def}
hyperbolic closing property means that we can find hyperbolic periodic points to shadow any segments.  For the sake of controlling  the Lyapunov exponents of the shadowing periodic points, we need some sources of hyperbolicity. We borrow some hyperbolicity from the hyperbolic periodic point $p$ of $CH^1(p)$.  The constant $a$ presents the proportion of hyperbolic periodic points in the periodic pseudo orbit. 
 \begin{Def}Let $f$ be a diffeomorphism on a compact manifold $M$ and let $K$ be  an compact invariant subset of $M$. Assume $E$ is an invariant one-dimensional subbundle of the tangent bundle $T_K(M)$.   We say that $f$ preserves an orientation of $E$ if there exists a continuous section of the bundle of orientations of $E$ over $K$, that is invariant by the action of $f$. In this case, the bundle $E$ can be identified to the trivial bundle $K\times \mathbb{R}$.
\end{Def}
\begin{Thm}\label{specification} Let $f$ be a diffeomorphism on a compact manifold $M$  with   $CH^1(p)$  .
If $f$ preserves the orientation of the subbundle $E^c$, both the closing property  and the hyperbolic closing property hold for $f|_{CH^1(p)}.$
\end{Thm}

One may have notice that there are some  shadowing lemmas for partially hyperbolic systems (such as central shadowing lemma, quasi-shadowing lemma in \cite{KT,HZZ}).  Crovisier \cite{C1} also gave a proof of shadowing
property for partially
 hyperbolic systems (with special trapped central model). 
 In his result, the shadowing points lie in the neighborhood of the targeted set, instead of in the  targeted set.   These shadowing lemmas are not suitable for our purposes.  
  In this paper, we always want 
 the shadowing points lie in $CH^1(p)$.   
 
 In section 5, we  state a corollary from Theorem \ref{specification} about the density of hyperbolic periodic measures  in the space of invariant measures.  We also give a construction of a maximal entropy measure.

\smallskip
 
{\it Notation:} Throughout this paper,  $W^{*}$ means the real stable manifolds or the real unstable manifolds, contrasting to the notation of plaques, $\mathcal{W}^{*}.$

\section{Preliminaries}
As described in the introduction, chain hyperbolic homoclinic classes have some properties parallel to the properties of Axiom A systems. 
In this section, we  introduce  some properties of chain hyperbolic homoclinic classes (following \cite{C2}). 
Let $H(p)$ be a chain hyperbolic homoclinic class in Definition \ref{chain}. Then, $H(p)$ has a dominated splitting $T_{H(p)}M=E^{cs}\oplus E^{cu}$.   The following lemma says that $H(p)$ contains a dense set of hyperbolic
periodic points with long stable and unstable manifolds. 
\begin{Lem}[Lemma 3.2 in \cite{C2}]\label{dense good periodic points}
For any $\delta>0$ small, there exists a dense set
$\mathcal{P}\subseteq H(p)$ of hyperbolic periodic points homoclinically
related to the orbit of $p$ with the following property.
\begin{itemize}
\item The modulus of the Lyapunov exponents of any point $q\in
\mathcal{P}$ are larger than $\delta$;
\item The plaques $\mathcal{W}^{cs}(q)$ and $\mathcal{W}^{cu}(q)$
for any point $q\in \mathcal{P}$ are  contained in the
stable manifolds and the unstable manifolds of $q$ respectively. 
\end{itemize}
\end{Lem}
For any hyperbolic periodic point $q\in\mathcal{P}$, we can denote $\mathcal{W}^{cs}(q)$ by $W^s(q)$ and denote $\mathcal{W}^{cu}(p)$ by $W^u(q)$. 
We can see from the  next lemma that chain-hyperbolic
homoclinic classes have  local product structures. This property is
crucial in the proof of shadowing property, both in the uniform case
and  this special partially hyperbolic case. It  motivated us the construction in Section 3.\begin{Lem}[Lemma 3.4 in \cite{C2}]\label{local product}  For any $\varepsilon>0$, there exists $\delta >0$ such that for any $x,y\in H(p)$, if $d(x,y)<\delta$, then $\mathcal{W}_{\varepsilon}^{cs}(x)$
and $\mathcal{W}_{\varepsilon}^{cu}(y)$ transversally intersect at a single point
$z\in H(p)$, where $\mathcal{W}_{\varepsilon}^{cs}(x)\subset \mathcal{W}^{cs}(x)$ is centered at $x$ with length $2\varepsilon$.  
\end{Lem}

\section{ The proof of Theorem \ref{specification}}
In this section we shall prove Theorem \ref{specification}.  
Let $f$ be a diffeomorphism on a compact manifold $M$  with a homoclinic class  $CH^1(p)$.
Assume that $f$ preserves the orientation of the subbundle $E^c$.  
Before heading to the main block of the proof, we clarify some notations and identify some constants along the way. 

From Definition \ref{CH^1(p)}, there is a splitting $$T_{CH^1(p)}M=E^{cs}\oplus E^u=E^{s}\oplus E^c\oplus E^{u}.$$ Thus, there is a plaque family $\mathcal{W}^{cs}$ associated with subbundle $E^{cs}$.  As the bundle $E^u$ here is uniformly expanding, there are unstable manifolds 
$W^u$ associated with $E^u$.   There is also a plaque family $\mathcal{W}^c$ associated with the bundle $E^c$.  The length of the circle plaques along $E^c$ can not be arbitrarily small. Otherwise, the bundle $E^c$ at the accumulation point of the circles with radii going to zero can not be well defined.  Thus, we can assume the radius of the  plaque family $\mathcal{W}^c$ are so small  that  the plaques for $\mathcal{W}^c$ are intervals, instead of   circles.

 Let $\varepsilon>0$ be small enough.  According  to the local product property (Lemma \ref{local product}),  there exists
$\frac{\delta}{2}>0$, such that
$$\mathcal{W}^{cs}_{\varepsilon}(x)\cap
W^u_{\varepsilon}(y)\neq\emptyset,\,\,\text{ for any } x,y\in CH^1(p) \text{ satisfying } d(x,y)\leq\frac{\delta}{2}.$$
Applying Lemma \ref{local product} again, we obtain
$0<\delta_1<\frac{\delta}{2}$ such that
$$\mathcal{W}^{cs}_{\frac{\delta}{2}}(x)\cap
W^u_{\frac{\delta}{2}}(y)\neq\emptyset,\,\,\text{ for any } x,y\in CH^1(p)\text{ satisfying } d(x,y)\leq\delta_1.$$
Due to Lemma \ref{dense good periodic points},
for $\delta_1>0$ given above, there exists   a finite collection of hyperbolic periodic points  $$\mathcal{C}(\varepsilon)=\{p_i\}_{i=1}^{N}\subset \mathcal{P}$$ such that for any $x_1, x_2\in CH^1(p),$  there exists $\{y_1,y_2,\ldots, y_k\}\subset \mathcal{C}(\varepsilon)$ such that 
$d(x_1,y_1)\leq \frac{\delta_1}{2}, d(y_i,y_{i+1})\leq \frac{\delta_1}{2}, , \,\forall 1\leq i\leq k-1$ and $d(y_k, x_2)\leq \frac{\delta_1}{2}.$  

Let $\lambda_0$ be the uniform lower bound of the absolute values of Lyapunov exponents of the hyperbolic periodic points in $\mathcal{P}$. For any hyperbolic periodic point $q \in\mathcal{P}$, we can denote $\mathcal{W}^{cs}(q)$
by $W^s(q)$.    Thus for $\varepsilon>0$ small enough  and any point
$p_k\in\mathcal{C}(\varepsilon)$, there exist constant $C>1$ and integers $K(p_k)$ such that
$$d(f^{K(p_k)P(p_k)}(y_1),f^{K(p_k)P(p_k)}(y_2))\leq Ce^{-\lambda_0 K(p_k)P(p_k)}d(y_1,y_2),\ \ \ \forall y_1,y_2\in W^{s}_{\varepsilon}(p_k),$$
$$d(f^{K(p_k)P(p_k)}(y_1),f^{K(p_k)P(p_k)}(y_2))\geq C^{-1}e^{\lambda_0 K(p_k)P(p_k)}d(y_1,y_2),\ \ \ \forall y_1,y_2\in W^{u}_{\varepsilon}(p_k),$$
and $$Ce^{-\lambda_0K(p_k)P(p_k)}\varepsilon\leq \frac{\delta_1}{2}\leq \frac{\delta}{2},$$
where $P(p_k)$ is the period of $p_k.$
Denote $$M(\varepsilon)=\sum_{i=0}^NK(p_i)P(p_i).$$
 By the definition of $K(\cdot)$ and $P(\cdot)$, we can see that $M(\varepsilon)$ only depends on $\varepsilon$. 
 
\begin{proof}[Proof of the closing property in Theorem \ref{specification}] We shall construct a sequence of points that are the intersection points of some unstable manifolds and $\mathcal{W}^{cs} $ plaques.  Fix  two points  in the same orbit  $x, f^n(x)\in CH^1(p)$ and two points in the same orbit $y,f^m(y)\in CH^1(p)$ arbitrarily. From the pick of $\mathcal{C}(\varepsilon)$, 
there is a subset $\{y_1,y_2,y_3,\ldots,y_{N_1}\}\subset\mathcal{C}(\varepsilon)$ satisfying the
following properties,
$$d(f^n(x),y_1)\leq\frac{\delta_1}{2},\,\,d(y_i,y_{i+1})\leq\frac{\delta_1}{2},\text{ and } d(y_{N_1},y)\leq\frac{\delta_1}{2}, \,\,\forall 1\leq i\leq N_1-1.$$
Similarly, there is a subset $\{y_1',y_2',y_3',\ldots,y_{N_2}'\}\subset\mathcal{C}(\varepsilon)$ satisfying the
following properties,
$$d(f^m(y),y_1')\leq\frac{\delta_1}{2},\,\,d(y_i',y_{i+1}')\leq\frac{\delta_1}{2},\text{ and } d(y_{N_2}',x)\leq\frac{\delta_1}{2}, \,\,\forall 1\leq i\leq N_2-1.$$

Then,  
$$W^{u}_{\delta}(f^n(x))\cap W^s_{\delta}(y_1)\neq\emptyset,\,\,W^u_{\delta}(y_i)\cap W^s_{\delta}(y_{i+1})\neq\emptyset,$$ $$\text{ and }
W^u_{\delta}(y_{N_1})\cap
\mathcal{W}^{cs}_{\delta}(y))\neq\emptyset,\,\,\forall 1\leq i\leq N_1-1.$$ 
Similarly, 
$$W^{u}_{\delta}(f^m(y))\cap W^s_{\delta}(y_1')\neq\emptyset,\,\,W^u_{\delta}(y_i')\cap W^s_{\delta}(y_{i+1}')\neq\emptyset,$$ $$\text{ and }
W^u_{\delta}(y_{N_2}')\cap
\mathcal{W}^{cs}_{\delta}(x))\neq\emptyset,\,\,\forall 1\leq i\leq N_2-1.$$ 

Let $M_1(\varepsilon)=\sum_{k=1}^{N_1}K(y_k)P(y_k)$ and $M_2(\varepsilon)=\sum_{k=1}^{N_2}K(y_k')P(y_k')$. We have $\max\{M_1(\varepsilon),M_2(\varepsilon)\}\leq M(\varepsilon). $
Now we  construct an one-sided periodic pseudo orbit $\mathcal{O}$ as follows,
\begin{eqnarray*}
&&\left \{ x, f(x), \ldots,f^{n-1}(x),\right. \\[2mm]
&&\,\,\,\underbrace{y_1,f(y_1),\ldots,f^{P(y_1)-1}(y_1)}_{K(y_1)},\\[2mm]
&&\,\,\,\,\,\,\,\,\,\,\,\,\,\,\,\,\,\,\,\,\,\,\,\,\,\,\,\cdots,\\[2mm]
&&\,\,\,\underbrace{y_{N_1},f(y_{N_1}),\ldots,f^{P(y_{N_1})-1}(y_{N_1})}_{K(y_{N_1})},\\
&&\,\,\, y, f(y), \ldots,f^{m-1}(y),\\[2mm]
&&\,\,\,\underbrace{y_1',f(y_1'),\ldots,f^{P(y_1')-1}(y_1')}_{K(y_1')},\\[2mm]
&&\,\,\,\,\,\,\,\,\,\,\,\,\,\,\,\,\,\,\,\,\,\,\,\,\,\,\,\cdots,\\[2mm]
&&\,\,\,\underbrace{y_{N_2}',f(y_{N_2}'),\ldots,f^{P(y_{N_2}')-1}(y_{N_2}')}_{K(y_{N_2}')},\\
&& x, f(x), \ldots,f^{n-1}(x), \\[2mm]
&&\big.\,\,\,\,\,\,\,\,\,\,\,\,\,\,\,\,\,\,\,\,\,\,\,\,\,\,\,\,\,\cdots\Big\}.
\end{eqnarray*}
{\it Claim 1:} There is a point $\tilde{z}\in CH^1(p)$ $\varepsilon$-shadowing $\mathcal{O}$. 
\smallskip

{\bf Step 1:} Construct $z_0,z_1,\ldots,z_{N_1}.$
Since $d(f^n(x), y_1)\leq \frac{\delta}{2}$, we obtain $z_{0}\in W^{u}_{\varepsilon}(f^n(x))\cap W^{s}_{\varepsilon}(y_1).$
 Combining $d(f^{K(y_1)P(y_{1})}(z_{0}),y_{1})<Ce^{-\lambda_0
K(y_1)P(y_1)}\varepsilon<\frac{\delta}{2}$ and $d(y_1,y_2)<\frac{\delta}{2},$ we deduce
 $d(f^{K(y_1)P(y_1)}(z_{0}), y_{2})<\delta.$ Thus we  obtain $$z_{1}\in W^{u}_{\varepsilon}(f^{K(y_1)P(y_1)}(z_{0}))\cap
W^{s}_{\varepsilon}(y_2).$$ Preceding in this fashion, we obtain $$z_{k}\in W^{u}_{\varepsilon}(f^{K(y_k)P(y_k)}(z_{k-1}))\cap
W^{s}_{\varepsilon}(y_{k+1}),$$ for any $1\leq k\leq N_1-1$.  
Then $d(f^{K(y_{N_1})P(y_{N_1})}(z_{N_1-1}), y_{N_1})\leq C e^{-\lambda
K(y_{N_1})P(y_{N_1})}\varepsilon<\frac{\delta_1}{2}$.   Observing also that $d(y_{N_1},y)\leq \frac{\delta_1}{2}$, we have $$d(f^{K(y_{N_1})P(y_{N_1})}(z_{N_1-1}),y)\leq \delta_1.$$
Thus we obtain $$z_{N_1}\in W^u_{\frac{\delta}{2}}(f^{K(y_{N_1})P(y_{N_1})}(z_{N_1-1}))\cap
\mathcal{W}^{cs}_{\frac{\delta}{2}}(y).$$

{\bf Step 2:}  Construct $z_{N_1+1},\ldots,z_{N_1+m}.$
 It follows from the thin trapped property of plaques $\mathcal{W}^{cs}$ that $f^{m}(z_{N_1})\in \mathcal{W}^{cs}_{\frac{\delta}{2}}(f^m(y)).$ (Here we may need a little adjustment of the constant $\delta_1$. This is because the family $\{\phi_x\}$ in the definition of thin trapped, which indicates the radius of contracted region, may not be a constant function. But from the continuity of the family $\{\phi_x\}$ and the compactness of $CH^1(p)$, after a little adjustment of the constants $\delta_1$, we still can say that $f^{m}(z_{N_1})\in \mathcal{W}^{cs}_{\frac{\delta}{2}}(f^m(y))$.) 
 Since $d(f^m(y), y_1')\leq \frac{\delta}{2}$, we have 
 $$d(f^m(z_{N_1}), f^m(y))\leq \delta.$$  
 Thus we obtain $$z_{N_1+m} \in
W^{u}_{\varepsilon}(f^m(z_{N_1}))\cap W^{s}_{\varepsilon}(y_1').$$

{\bf Step 3:} Construct $z_{N_1+m+N_2}$. 
Applying the construction of step 1, we obtain 
$$z_{N_1+m+k}\in W^{u}_{\varepsilon}(f^{K(y_k')P(y_k')}(z_{N_1+m+k-1}))\cap
W^{s}_{\varepsilon}(y_{k+1}'),$$ for any $1\leq k\leq N_2$.

{\bf Step 4:} Construct $z_{N_1+m+N_2+n}$.
Applying the construction of step 2, we obtain 
$$z_{N_1+m+N_2+k}\in W^{u}_{\varepsilon}(f^{K(y_k)P(y_k)}(z_{N_1+m+N_2+k}))\cap
W^{s}_{\varepsilon}(y_{k+1}),$$ for any $1\leq k\leq n$.

{\bf Step 5:} Construct $z_{k(N_1+m+N_2+n)}$.
Repeating the above steps, we obtain 
$z_{k(N_1+m+N_2+n)}$.

Take an accumulation point $\tilde{z}$ of
$\{f^{-(k(N_1+m+N_2+n))}(z_{k(N_1+m+N_2+n)})\}_{k\in\mathbb{N}}$. 
From the choice of  $z_{k(N_1+m+N_2+n)}$, we deduce that $\tilde{z}$ is $\varepsilon$-shadowing the pseudo-orbit $\mathcal{O}$. 

\smallskip
{\it Claim 2:} There exists a periodic point $z\in CH^1(p)$ $\varepsilon$-shadowing $\mathcal{O}$.  

\smallskip
Denote $f^{N_1+m+N_2+n}$ by $g$.  By the uniformly expanding property of  $E^u$ bundle, we deduce that  $\{g^k(\tilde{z})\}_{k=1}^{\infty}\subset
\mathcal{W}^{cs}_{2\varepsilon}(\tilde{z})$. Let $z$ be an accumulation point for $\{g^{k}(\tilde{z})\}_{k=1}^{\infty}$.   
Due to the uniform contraction of $E^s$ bundle,  we have  $z\in\mathcal{W}^c_{2\varepsilon}(\tilde{z})$. 

It remains to show that $z$ is a periodic point for  $g$. Some arguments of one dimensional dynamics allow us to obtain this property. 
First of all, we define a projection along the $E^s$ direction. Since $E^s$ bundle is uniformly contracted, there exist well placed foliations tangent to $E^s$. Let $\pi: \mathcal{W}^{cs}(\tilde{z})\rightarrow \mathcal{W}^c(\tilde{z})$ be the projection map along the foliations tangent to $E^s.$  From the invariance of $E^s$, we have $\pi\circ g^2=g\circ \pi \circ g.$
Secondly, we give an orientation for the interval $\mathcal{W}^c_{2\varepsilon}(\tilde{z}).$ Without loss of generality, assume that we have interval $[\tilde{z}, g(\tilde{z}))].$  Then we have the interval $ \pi\circ g[\tilde{z}, \pi\circ g(\tilde{z})]=[\pi\circ g(\tilde{z}), \pi\circ g^2(\tilde{z})],$  or $ \pi\circ g[\tilde{z}, \pi\circ g(\tilde{z})]=[\pi\circ g^2(\tilde{z}), \pi\circ g(\tilde{z})].$  
 Since $f$ preserves the central bundle,  we are reduced to the case  $$ \pi\circ g[\tilde{z}, \pi\circ g(\tilde{z})]=[\pi\circ g(\tilde{z}), \pi\circ g^2(\tilde{z})].$$
  Combining the fact that $\mathcal{W}^c_{2\varepsilon}(\tilde{z})$ is an interval with bounded length and the assumption that $f$ preserves the orientation of $E^c$,  we have that $l(\pi\circ g^k[\tilde{z}, \pi\circ g(\tilde{z}))])$ goes to zero when $k$ goes to $+\infty$.  Thus $z$ is a fixed point for $g$.   
   On account of   the fact that $CH^1(p)$ is a closed and invariant set, we have $z\in CH^1(p).$ 
 We complete   the proof of the closing property.  

\end{proof}
\begin{figure}
\begin{center}
\setlength{\unitlength}{1.2mm}
 \begin{picture}(60,40)
 \linethickness{1pt}

\qbezier(0,20)(10,40)(30,30)\qbezier[20](0,20)(10,5)(30,10)\qbezier[20](30,30)(40,25)(30,10)
\cbezier[500](30,10)(40,10)(45,0)(30,10)\put(36,24){$\leq M(\varepsilon)$}\put(33,12){$p$}\put(38,10){$\frac{n(1-a)}{a}$}
\put(2,22){$x$}\put(32,32){$f^{n}(x)$}\put(10,5){$\leq M(\varepsilon)$}
\put(12,-3){Picture 1.}
\end{picture}
\end{center}
\end{figure}
\newpage
\begin{proof}[Proof of the hyperbolic closing property in Theorem \ref{specification}] Fix any two points in the same orbit $x,$ $f^n(x)$.  Without
loss of generality, we can assume  the hyperbolic point $p$ for $CH^1(p)$ to be a fixed point.  We will construct an one-sided periodic pseudo orbit, which goes through not only the segment $\{x,\ldots, f^n(x)\}$ but also the hyperbolic fixed point $p$. We will force the pseudo orbit to stay enough
 time at $p$ in order to get enough hyperbolicity for the corresponding shadowing periodic point.  See picture 1 for the segments and the corresponding periodic pseudo orbit.

Fix  two points in the same orbit $x, f^n(x)\in CH^1(p)$  arbitrarily. 
There is a subset $\{y_1,y_2,y_3,\ldots,y_{N_1}\}\subset\mathcal{C}(\varepsilon)$ satisfying the
following properties,
$$d(f^n(x),y_1)\leq\frac{\delta_1}{2},\,\,d(y_i,y_{i+1})\leq\frac{\delta_1}{2},\text{ and }d(y_{N_1},p)\leq\frac{\delta_1}{2}.$$
Then,
$$W^{u}_{\delta}(x)\cap W^s_{\delta}(y_1)\neq\emptyset,\,\,W^u_{\delta}(y_i)\cap W^s_{\delta}(y_{i+1})\neq\emptyset,\text{ and }
W^u_{\delta}(y_{N_1})\cap
\mathcal{W}^{cs}_{\delta}(f^n(x)))\neq\emptyset.$$ 
There is another subset $\{y_1',y_2',y_3',\ldots,y_{N_2}'\}=\mathcal{C}(\varepsilon)$ satisfying the
following properties,
$$d(p,y_1')\leq\frac{\delta_1}{2},\,\,d(y_i',y_{i+1}')\leq\frac{\delta_1}{2},\text{ and }d(y_{N_2}',x)\leq\frac{\delta_1}{2}.$$
Then,
$$W^{u}_{\delta}(p)\cap W^s_{\delta}(y_1')\neq\emptyset,\,\,W^u_{\delta}(y_i')\cap W^s_{\delta}(y_{i+1}')\neq\emptyset,\text{ and }
W^u_{\delta}(y_{N_2}')\cap
\mathcal{W}^{cs}_{\delta}(x))\neq\emptyset.$$ 
Let $M_1(\varepsilon)=\sum_{k=1}^{N_1}K(y_k)P(y_k)$  and $M_2(\varepsilon)=\sum_{k=1}^{N_2}K(y_k')P(y_k')$ . We have $M_1(\varepsilon), M_2(\varepsilon)\leq M(\varepsilon). $

Now we construct an one-sided periodic pseudo orbit $\mathcal{O}'$ as follows,
\begin{eqnarray*}
&& \Big\{ x, f(x), \ldots,f^{n}(x),\Big. \\[2mm]
&&\,\,\,\underbrace{y_1,f(y_1),\ldots,f^{P(y_1)-1}(y_1)}_{K(y_1)},\\[2mm]
&&\,\,\,\,\,\,\,\,\,\,\,\,\,\,\,\,\,\,\,\,\,\,\,\,\,\,\,\cdots,\\[2mm]
&&\,\,\,\underbrace{y_{N_1},f(y_{N_1}),\ldots,f^{P(y_{N_1})-1}(y_{N_1})}_{K(y_{N_1})},\\
&&\,\,\,\underbrace{ p, p, \ldots, p}_{\frac{an}{1-a}},\\[2mm]
&&\,\,\,\underbrace{y_1',f(y_1'),\ldots,f^{P(y_1')-1}(y_1')}_{K(y_1')},\\[2mm]
&&\,\,\,\,\,\,\,\,\,\,\,\,\,\,\,\,\,\,\,\,\,\,\,\,\,\,\,\cdots,\\
&&\,\,\,\underbrace{y_{N_2}',f(y_{N_2}'),\ldots,f^{P(y_{N_2}')-1}(y_{N_2}')}_{K(y_{N_2}')},\\
&&\,\,\,\,x, f(x), \ldots,f^{n}(x), \\[2mm]
&&\Big.\,\,\,\,\,\,\,\,\,\,\,\,\,\,\,\,\,\,\,\,\,\,\,\,\,\,\,\,\,\cdots\Big\}.
\end{eqnarray*}

Applying the same arguments in the proving of the closing lemma property, we deduce that there exists $q\in CH^1(p)$ with period $P(q)=\frac{n}{1-a}+k$, where $0\leq k\leq 2M(\varepsilon)$,  which is  $\varepsilon$-shadowing the pseudo orbit $\mathcal{O}'$. 

It remains to show that the Lyapunov exponent along the central direction of $q$ is smaller than $-\frac{a\lambda}{4}$.  In order to carry out  some useful estimates, we state a claim first.

{\it Claim 3:} 
For any number $\mu>0$, there exists integer $N>0$, such that for any $n\geq N$, we have 
$$\|Df^n_x|_{E^c}\|\leq e^{\mu n},\,\,\, \forall x\in CH^1(p).$$ 

 Suppose the claim does not hold, then there exists a number $\mu>0$ satisfying the following:  for any $k>0$ there is $n_k\geq k$ such that  there exists a point $x_k\in CH^1(p)$, such that 
$$\|Df^{n_k}_{x_k}|_{E^c}\|> e^{\mu n_k} .$$
Then, for the accumulation point $x$ of $x_k$, we have the Lyapunov exponent along the $E^c$ direction at point $x$ is greater than $\mu>0$, which is a contradiction to the property of thin trapped. Thus, we finish the proof of Claim 3.

For $a>0$, assume  $\mu<\frac{a\lambda}{4(1-a)}$. 
Assuming $n\geq N$ and $\varepsilon$ small enough, we have 
\begin{eqnarray*}
\|Df^{\frac{n}{1-a}+k}(q)\mid_{E^{c}}\|&\leq&\|Df^{n}_{q}\mid_{E^{c}}\|\cdot\|Df^{\frac{an}{1-a}+k}_{f^{n}(q)}|_{E^c}\| \\
&\leq& e^{n\mu}\cdot\|Df^{\frac{an}{1-a}+k}_{f^{n}(q)}|_{E^c}\|\\
&\leq& e^{n\mu}\cdot e^{\varepsilon \frac{an}{1-a}}\cdot\|Df^{\frac{an}{1-a}}_{p}|_{E^c}\|\\
&\leq&e^{n\mu+\varepsilon \frac{an}{1-a}}\cdot e^{-\lambda \frac{an}{1-a}}\\
&\leq&e^{ ((1-a)\mu-a\lambda+\varepsilon)(\frac{n}{1-a}+k)}\\
&\leq&e^{((1-a)\mu-\frac{a}{2}\lambda)(\frac{n}{1-a}+k)}\\
&\leq&e^{-\frac{a\lambda}{4}(\frac{n}{1-a}+k)},
\end{eqnarray*}
where the third inequality comes from the fact that $f^n(q)$ is close to $p$ and the fact that $k$ only depends on $\varepsilon$ and thus can be controlled when $n$ is large enough.  
We complete the proof of the hyperbolic closing property. 
\end{proof}

\section{Exponential growth of hyperbolic points and entropy} We adopt the notation  $P_n^a(f)$ given in the introduction with a little adjustment. 
Let
\begin{eqnarray*}  P^a_n(f|_{CH^1(p)})  &=&\{x\in CH^1(p)|f^n(x)=x, \lambda(x)\geq \frac{a\lambda}{4}>0,\\
&& \text{ where } \lambda(x) \text{ satisfies } \|Df^{n}|_{E^{c}(f^i(x))}\|= e^{-\lambda(x)n}\},
\end{eqnarray*}
where $\lambda$ is the largest negative Lyapunov exponent  of the hyperbolic periodic points $p$.

The following observation says if the central bundle  is one dimensional, then for any orbit in the collection of hyperbolic periodic orbits with  Lyapunov exponents uniformly bounded away from zero,  there exists a point in this orbit with uniform size of stable manifold and unstable manifold.     
The size  of these  stable manifolds and unstable manifolds  is irrelevant to the period.  It is only relevant  to the bound of Lyapunov exponents.  This result has been given in Lemma 3.3 of \cite{C2}. 
\begin{Lem}\label{Separated}\cite{C2}Assume that there is a dominated splitting $E^s\oplus E^c\oplus E^u$ for a diffeomorphism $f$ on a compact invariant set $K$.  Assume the dimension of $E^c$ is $1$. Then, for any $\delta>0$, there exists $\rho>0$ with the following property: Let $\mathcal{O}\subset K$ be a periodic orbit whose Lyapunov exponents along $E^s\oplus E^c$ are smaller than $-\delta$. Then there exists $q\in\mathcal{O}$ whose stable manifolds contains the plaque at $q$ along $E^s\oplus E^c$ direction with radius $\rho$.   
\end{Lem}
\begin{Rem}The Pliss time given by the Pliss lemma \cite{P} for every hyperbolic periodic orbit $\mathcal{O}$ can be the $q$ given in Lemma \ref{Separated}.
\end{Rem}

\begin{Cor}\label{expansive contants} For $0<a<1$, there exists a constant $ \xi(a) $ such that for any $n\in \mathbb{N}^{+}$ and any $x,y\in
  P^a_n(f|_{CH^1(p)})  $ with $d(f^i(x),f^i(y))\leq  \xi(a) , \forall i=1,\dots,n$, then
$x=y$. Moreover, $\sharp   P^a_n(f|_{CH^1(p)})  <\infty.$
\end{Cor}

\begin{Def} A subset $s(n,\varepsilon)\subset CH^1(p)$ is called  a {\it  $(n,\varepsilon)$-separated set} if for any two points  $x\neq y\in s(n,\varepsilon)$, there exists $i\in [0,n-1]$ such that
$$d(f^i(x),f^i(y))\geq \varepsilon.$$  
\end{Def}

Applying  Corollary \ref{expansive contants}, one deduces that
$$\sharp   P^a_n(f|_{CH^1(p)})  \leq S(n, \xi(a) ),$$ where $S(n, \xi(a) )$
denotes the number of  separated set $s(n, \xi(a) )$ with  maximal cardinal.  For any small constant $\varepsilon$,
$h_{\text{top}}(f|_{CH^1(p)})=\limsup_{n\rightarrow +\infty}1/n\log \sharp S(n, \varepsilon ).$
 So, we have 
\begin{eqnarray}\label{lower}\limsup_{n\rightarrow +\infty}\frac{1}{n}\log \sharp   P^a_n(f|_{CH^1(p)})  \leq \limsup_{n\rightarrow +\infty}\frac{1}{n}\log \sharp S(n, \xi(a) )\leq h_{\text{top}}(f|_{CH^1(p)}).\end{eqnarray}

Now we start to prove that the exponential growth rate of the  number of hyperbolic periodic points on $CH^1(p)$ is sufficient  to control  the topological entropy on $CH^1(p)$.

\begin{Prop}\label{exponential control}Let $f$ be a diffeomorphism on a compact manifold $M$ with  $CH^1(p)$. For some $a>0$, let
\begin{eqnarray*}  P^a_n(f|_{CH^1(p)})  &=&\{x\in CH^1(p)|f^n(x)=x, \lambda(x)\geq \frac{a\lambda}{4}>0,\\
&& \text{ where } \lambda(x) \text{ satisfies } \|Df^{n}|_{E^{c}(f^i(x))}\|= e^{-\lambda(x)n}\},
\end{eqnarray*}
where $\lambda$ is the largest negative Lyapunov exponent  of the hyperbolic periodic points $p$. 
Then, we have \begin{eqnarray}(1-a)h_{\text{top}}(f|_{CH^1(p)})\leq\limsup_{n\rightarrow +\infty}\frac{1}{n}\log \sharp
  P^a_n(f|_{CH^1(p)})  .\end{eqnarray}
\end{Prop}
\begin{proof}
We assume that $f$ preserves the orientation of $E^c$. Otherwise, consider  $f^2$.  
By  Theorem  \ref{specification},  for  any $a>0$ and any $\varepsilon>0$  small enough, there
exists $M(\varepsilon)\in\mathbb{N}$ such that for any two points in the same orbit,  $x, f^n(x)\in CH^1(p)$,  there is a hyperbolic periodic point $z\in CH^1(p)$ with period $P(z)$ such that 
\begin{enumerate}
\item the Lyapunov exponent along the central direction $E^c$ of periodic point $z$ is smaller than $-\frac{a\lambda}{4};$
\item  $d(f^i(z), f^i(x))\leq \varepsilon, \forall\, 0\leq i\leq n$ and  $\frac{n}{1-a}\leq P(z)\leq \frac{n}{1-a}+2M(\varepsilon).$ 
\end{enumerate}
Fix an arbitrarily positive constant $a>0$.  Let $n$ large enough such that $\frac{2M(\varepsilon)}{\frac{n}{a}+2M(\varepsilon)}$ is small enough. 
So, for any point $x\in s((1-a)n,4\varepsilon)$, there exists a hyperbolic periodic point $z\in P^a_{n+k}(f|_{CH^1(p)})$, which can $\varepsilon$-shadow the segment $\{x, f^n(x)\}$.  Thus, we have $$S((1-a)n,\,\,4\varepsilon)\leq \frac{1}{M(\varepsilon)}\sum_{i=n}^{i=n+M(\varepsilon)} \sharp
  P^a_{i}(f|_{CH^1(p)}). $$ 
Moreover, we have  $$\limsup_{n\rightarrow
\infty}\frac{1}{n}\log \sharp S((1-a)n,\,\,4\varepsilon)\leq
\limsup_{n\rightarrow +\infty}\frac{1}{n}\log \sharp   P^a_{n}(f|_{CH^1(p)})  .$$
Finally, we have \begin{eqnarray*} \label{dai a }
(1-a)h_{\text{top}}(f|_{CH^1(p)})&=&\lim_{\varepsilon\rightarrow 0}\limsup_{n\rightarrow
\infty}\frac{1}{n}\log \sharp
S(an,\,\,8\varepsilon)\\
&\leq&\limsup_{n\rightarrow
+\infty}\frac{1}{n}\log \sharp   P^a_{n}(f|_{CH^1(p)})  .
\end{eqnarray*}

\end{proof}
\begin{proof}[Proof of Theorem \ref{growth of periodic points }]
From Proposition \ref{exponential control} and equation \eqref{lower},  we deduce $$h_{\text{top}}(f|_{CH^1(p)})= \lim_{a\rightarrow 0}\limsup_{n\rightarrow
+\infty}\frac{1}{n}\log \sharp   P^a_n(f|_{CH^1(p)}).$$
\end{proof}

\section{Density of Periodic measures in the space of invariant measures}
The prevalence of hyperbolic periodic measures in the space of invariant measures is also very interesting.  K. Sigmund in the seventies proved the following results for  Axiom A diffeomorphisms (using the specification property):
\begin{Thm}\cite{S}Let $\Lambda$ be a basic set of an Axiom A diffeomorphism $f$. Denote by $\mathcal{M}_f(\Lambda)$  the set of invariant probabilities supported on $\Lambda$. Then, we have :
\begin{enumerate}
\item every $\mu\in\mathcal{M}_f(\Lambda)$ is approximated by hyperbolic periodic measures $\mu_n\in \mathcal{M}_f(\Lambda)$;
\item the generic measure $\mu\in\mathcal{M}_f(\Lambda)$ is ergodic;
\item the generic measure $\mu\in\mathcal{M}_f(\Lambda)$ has entropy zero: $h_{\mu}(f)=0.$
\end{enumerate}
\end{Thm}  

Arguing in the same fashion as in \cite{S}, {\it mutatis mutandis}, we have the following corollary from the closing property in Theorem \ref{specification}. 
\begin{Cor}\label{distribution of periodic measures}Let $f$ be a diffeomorphism on a compact manifold $M$  with  $CH^1(p)$. Denote by $\mathcal{M}_f$ the set of invariant probabilities supported on $CH^1(p)$. If $f$ preserves the orientation of the subbundle $E^c$, then we have :
\begin{enumerate}
\item every $\mu\in\mathcal{M}_f$ is approximated by hyperbolic periodic measures $\mu_n\in \mathcal{M}_f$;
\item the generic measure $\mu\in\mathcal{M}_f$ is ergodic;
\item the generic measure $\mu\in\mathcal{M}_f$ has entropy zero: $h_{\mu}(f)=0.$
\end{enumerate}
\end{Cor}  
 
About the existence of maximal entropy measures, there is a general result by Liao, Viana and Yang \cite{LVY}  that for any diffeomorphisms far away from tangency, there is an invariant probability measure with maximal entropy.
Combining  Corollary \ref{distribution of periodic measures} and the existence of maximal entropy measures, we get  the existence of  a maximal entropy measure that can be approximated by a sequence of measures supported  on hyperbolic periodic points for  $f|_{CH^1(p)}.$ 
\begin{Cor} \label{measures} Let $f$ be a diffeomorphism on a compact manifold $M$  with  $CH^1(p)$. Then, there exists an invariant measure $\mu$ supported  on $CH^1(p)$, such that 
$$h_{\mu}(f|_{CH^1(p)})=h_{\text{top}}(f|_{CH^1(p)})$$
and there exist $\mu_n$ supported on hyperbolic periodic points in $CH^1(p)$ such that   $\mu_n\rightarrow\mu$ in weak $*$ topology,  as $n\rightarrow +\infty$.  
\end{Cor}

Here we prefer to  construct a maximal invariant measure  from Theorem \ref{growth of periodic points }.  
From Proposition \ref{exponential control}, we can find a
subsequence $\{n_i\}_{i=0}^{\infty}$ (corresponding to $a$), such
that
\begin{eqnarray*}ah(f|_{CH^1(p)})\leq\lim_{i\rightarrow +\infty}\frac{1}{n_i}\log \sharp
P^a_{n_i}(f|_{CH^1(p)}).\end{eqnarray*}
Denote $$\mu^a_i=\frac{1}{\sharp P^a_{n_i}(f|_{CH^1(p)})}\sum_{x\in
 P^a_{n_i}(f|_{CH^1(p)})}\delta_x.$$
 Then, there exists a  subsequence
$\{i_k\}_{k=1}^{\infty}$ such that
$\mu^a= \lim_{k\rightarrow\infty}\mu^a_{i_k}$ and a
subsequence $\{a_k\}_{k=1}^{\infty}$ such that
$\mu=\lim_{k\rightarrow\infty}\mu^{a_k},$
where $a_k\rightarrow 1$ as $k\rightarrow\infty.$

\begin{Lem} \label{sequences}There exist a sequence $a_k$, which goes to  $ 0$ as $k\rightarrow+\infty$, and  sequences   $n_i(a_k)$, which  go to $\infty$ as $i\rightarrow +\infty$, satisfying the following $$
h_{\mu^{a_k}}(f|_{CH^1(p)})\geq\lim_{i\rightarrow\infty}\frac{1}{n_{i}(a_k)}\log\sharp
P^{a_k}_{n_{i}(a_k)}(f|_{CH^1(p)}).$$
\end{Lem}
 \begin{proof}
  It is a classical procedure similar
to the proof of the variational principle.
Let $\delta(a_k)$ be the expansive constants given in lemma
(\ref{expansive contants}). We may choose a finite partition
$\xi=\{A_1,\cdots, A_m\}$ such that $\diam (A_j)<\delta(a_k)$ and
$\mu^{a_k}(\partial A_j)=0$ for each $j$. Thus each member of
$\bigvee_{i=0}^{n_i(a_k)-1}f^{-i}\xi$ contains at most one point in
$P_{n_{i}(a_k)}^{a_k}(f)$. So\, $\sharp P_{n_{i}(a_k)}^{a_k}(f)$
members of $\bigvee_{i=0}^{n_i(a_k)-1}f^{-i}\xi$ each has
$\mu^{a_k}_{i}-$measure $1/ \sharp P_{n_{i}(a_k)}^{a_k}(f)$ and
others have $\mu^{a_k}_{i}-$measure zero. Fix
 natural number $q$ with $1<q<n_{i}(a_k)$. For $0\leq j\leq q-1$, define
 $a(j)=[(n_{i}(a_k)-j)/q]$. We have
 the following facts:
 \begin{enumerate}
 \item $a(0)\geq a(1)\geq\cdots \geq a(q-1)$;
\item Fix $0\leq j \leq q-1$. Then
$$\{0,1,\cdots, n_{i}(a_k)-1\}=\{j+rq+i\,\,|\,0\leq r\leq a(j)-1, 0\leq i\leq q-1\}\cup S$$
where
$$S=\{0,1,\cdots,j-1\}\cup \{j+a(j)q,j+a(j)q+1,\cdots,n_{i}(a_k)-1\},\text{ and } \sharp S\leq2q.$$
 \end{enumerate}
\smallskip
From these facts  we have
$$\bigvee_{i=0}^{n_{i}(a_k)-1}f^{-i}\xi=\bigvee_{r=0}^{a(j)-1}f^{-(rq+j)}
( \bigvee_{i=0}^{q-1}f^{-i}(\bigvee_{i\in S}f^{-i}\xi)).$$
Therefore
\begin{eqnarray*}\log\sharp
P_{n_{i}(a_k)}^{a_k}(f|_{CH^1(p)})&=&H_{\mu^{a_k}_{i}}(\bigvee_{i=0}^{n_{i}(a_k)-1}f^{-i}\xi)\\[2mm]
  &\leq&\sum_{r=0}^{a(j)-1}H_{\mu^{a_k}_{i}}(\bigvee_{r=0}^{a(j)-1}f^{-(rq+j)}
 \bigvee_{i=0}^{q-1}f^{-i}\xi)+\sum_{i\in
 S}H_{\mu^{a_k}_i}(f^{-i}\xi)\\[2mm]
 &\leq&\sum_{r=0}^{a(j)-1}H_{\mu^{a_k}_{i}\circ
 f^{-(rq+j)}}(\bigvee_{i=0}^{q-1}f^{-i}\xi)+2q\log m.
\end{eqnarray*}
Summing this inequality over $j$ from 0 to $q-1$, we obtain 
\begin{eqnarray*}
q\log\sharp
P_{n_{i}(a_k)}^{a_k}(f|_{CH^1(p)})&\leq&\sum_{l=0}^{n_{i}(a_k)-1}H_{\mu^{a_k}_{i}}\circ
f^{-l}(\bigvee_{i=0}^{q-1}f^{-i}\xi)+2q^2\log
m\\
&=&n_{i}(a_k)H_{\mu^{a_k}_{i}}(\bigvee_{i=0}^{q-1}f^{-i}\xi)+2q^2\log
m.
\end{eqnarray*} Divided by $n_{i}(a_k)q$, it becomes
\begin{equation}\label{1}\frac{1}{n_{i}(a_k)}\log\sharp
P_{n_{i}(a_k)}^{a_k}(f|_{CH^1(p)})\leq\frac{1}{q}
H_{\mu^{a_k}_{i}}(\bigvee_{i=0}^{q-1}f^{-i}\xi)+\frac{2q}{n_{i}(a_k)}\log
m.
\end{equation}
Since for each $B\in\bigvee_{i=0}^{q-1}f^{-i}\xi$ ,
$\mu^{a_k}_{i}(\partial(B))=0$, we have
$\lim_{i\rightarrow\infty}\mu_{i}^{a_k}(B)=\mu_{a_k}(B)$. Thus,
$$\lim_{i\rightarrow\infty}H_{\mu_{i}^{a_k}}(\bigvee_{i=0}^{q-1}f^{-i}\xi)
=H_{\mu^{a_k}}(\bigvee_{i=0}^{q-1}f^{-i}\xi).$$ 
Let the $i$ in equation (\ref{1}) go to infinity and then let the $q$ in equation (\ref{1}) go to infinity.  We obtain 
$$\lim_{i\rightarrow\infty}\frac{1}{n_{i}(a_k)}\log\sharp
P_{n_i(a_k)}^{a_k}(f|_{CH^1(p)})\leq
h_{\mu^{a_k}}(f|_{CH^1(p)}).$$
\end{proof}
\begin{proof}[The proof of Corollary \ref{measures}]The measure $ \mu=\lim_{k\rightarrow\infty}\mu^{a_k}$
  given above can be approximated by measures supported on hyperbolic periodic points.
By the variational principle, $h_{\mu}(f|_{CH^1(p)})\leq h_{\text{top}}(f|_{CH^1(p)})$. Combining Lemma \ref{sequences}, Theorem \ref{growth of periodic points } and the upper semi
continuous of entropy map, we deduce  that $h_{\mu}(f|_{CH^1(p)})\geq
h_{\text{top}}(f|_{CH^1(p)}).$ Thus, we finish the proof of Corollary \ref{measures}.
\end{proof}


\section*{Acknowledgement}
We  would  like to  thank Sylvain Crovisier for his useful suggestions. We would like  to thank Gang Liao and Xiaodong Wang for their useful
conversations.  We also would like to thank Jinxin Xue for his careful reading. The second author also would like to thank China Scholarship Council  for its financial support.

 \end{document}